\documentclass[11pt]{amsart}
\usepackage{graphicx}
\usepackage{mathabx} 
\usepackage{oldgerm}
\usepackage{mathdots}  
\usepackage{stmaryrd}
\usepackage{bm}
\usepackage[all]{xy}
\usepackage{color}
\usepackage{enumerate}

\usepackage{hyperref}
\usepackage{mathrsfs}
\usepackage{tikz} 
\usetikzlibrary{arrows,%
  decorations.markings,%
  matrix,%
  shapes}
  
\newcommand{\diagram}[3]{\matrix (#1) [matrix of math nodes,row
  sep={#2},column sep={#3},text height=1.5ex,text
  depth=0.25ex]}

\usepackage[latin1]{inputenc}
\usepackage{amsmath,amsthm, amscd, amssymb, amsfonts}

\numberwithin{equation}{section}

\newcommand{\End}{\mbox{\rm End\,}}
\theoremstyle{plain}

\numberwithin{equation}{section}

\newtheorem{theorem}[equation]{Theorem}
\newtheorem{corollary}[equation]{Corollary}

\newtheorem{lemma}[equation]{Lemma}

\theoremstyle{definition}
\newtheorem*{definition}{Definition}

\newtheorem{example}[equation]{Example}

\theoremstyle{definition}
\newtheorem{remark}[equation]{Remark}

\newcommand{\C}{\mathscr{C}}
\newcommand{\M}{\mathscr{M}}
\newcommand{\N}{{\mathbb N}}
\newcommand{\Z}{{\mathbb Z}}

\newcommand\Ext{\operatorname{Ext}}

\newcommand\Homol{\operatorname{H}}
\newcommand\coh{\Homol}

\newcommand\cx{\operatorname{cx}}

\newcommand\op{\operatorname{op}}
\newcommand\ot{\otimes}
\renewcommand\mod{\operatorname{mod}}
\newcommand\stmod{\operatorname{\underline{mod}}}
\newcommand\Hom{\operatorname{Hom}}

\newcommand\rad{\operatorname{rad}}

\newcommand\Kdim{\operatorname{Kdim}}

\newcommand\VM{V_{\M}}
\newcommand\unit{\mathbf{1}}

\newcommand{\DOT}{\setlength{\unitlength}{1pt}\begin{picture}(2.5,2)(1,1)\put(2,3){\circle*{2}}\end{picture}}
\newcommand{\bu}{\DOT}

\DeclareMathOperator{\Vect}{Vec}
\DeclareMathOperator{\Ver}{Ver}
\newcommand{\Coh}{\operatorname{H}\nolimits}
\newcommand{\Ho}{\operatorname{\Coh^{\bu}}\nolimits}

\newcommand{\Spec}{\operatorname{Spec}\nolimits}

\newcommand{\az}{\mathfrak{a}}

\newcommand{\m}{\mathfrak{m}}
\newcommand{\p}{\mathfrak{p}}
\newcommand{\q}{\mathfrak{q}}
\newcommand{\rz}{\mathfrak{r}}

\makeatletter
\def\blx@maxline{77}
\makeatother

\begin{document}
\title[Representation type of a finite tensor category]
{On the representation type of a\\
finite tensor category}

\author[P.A.\ Bergh, K.\ Erdmann, J.Y.\ Plavnik, S.\ Witherspoon]{Petter Andreas Bergh, Karin Erdmann, Julia Yael Plavnik, \\
 Sarah Witherspoon}

\address{Petter Andreas Bergh \\ Institutt for matematiske fag \\ NTNU \\ N-7491 Trondheim \\ Norway} 
\email{petter.bergh@ntnu.no}
\address{Karin Erdmann \\ Mathematical Institute \\ University of Oxford \\ Andrew Wiles Building \\ Radcliffe Observatory Quarter \\ Oxford OX2 6GG \\ United Kingdom} 
\email{erdmann@maths.ox.ac.uk}
\address{Julia Yael Plavnik \\ Department of Mathematics \\ Indiana University \\ Bloomington \\ Indiana 47405 \\ USA \& Department of Mathematics and Data Science, Vrije Universiteit Brussel, Pleinlaan 2, 1050 Brussels, Belgium}
\email{jplavnik@iu.edu}
\address{Sarah Witherspoon \\ Department of Mathematics \\ Texas A \& M University \\ College Station \\ Texas 77843 \\ USA}
\email{sjw@tamu.edu}


\begin{abstract}
Given an exact module category over a finite tensor category with finitely generated cohomology, we show that if there exists an object of complexity at least three, then the category is of wild representation type. In particular, if the Krull dimension of the cohomology ring of the finite tensor category is at least three, then this category is wild.
\end{abstract}

\maketitle

\section{Introduction}

In 1990, Rickard~\cite{R} showed that representation type can be discernible via homological information, for example in finite group algebras. Almost twenty years later, Farnsteiner~\cite{F} realized this connection for finite dimensional cocommutative Hopf algebras using the support variety theory that had just been developed for these algebras. Specifically, he proved that such a Hopf algebra has wild representation type if its cohomology has Krull dimension at least three, or equivalently if it has a module of complexity at least three. The proof relies heavily on prior work of Drozd~\cite{D} and Crawley-Boevey~\cite{CB}, in particular the key role played by the Auslander-Reiten translate. Farnsteiner's techniques apply more generally to settings where cohomology is finitely generated and there is a well behaved support variety theory. Other authors used these techniques successfully for finite dimensional self-injective algebras~\cite{BS,FW2} and braided finite dimensional Hopf algebras~\cite{FW1}. These three papers require cohomological finiteness conditions to ensure a robust support variety theory. In~\cite{BS}, such a finiteness hypothesis is placed on the cohomology of a module, in~\cite{FW1} on Hopf algebra cohomology, and in~\cite{FW2} on Hochschild cohomology. In these papers, the results are applied to certain small quantum groups, Hecke algebras, Nichols algebras, and more. The theory has been further adapted to prove representation type results for Lie superalgebras~\cite{L}. 

In this paper, we transport these methods to finite tensor categories and their finite module categories, applying the support variety theory developed in~\cite{BPW1,BPW2}. Our main result is Theorem~\ref{thm:main}, stating that under a cohomological finite generation condition, such a module category is of wild representation type if it has an object of complexity at least three. Here, we define wild representation type in Section~\ref{sec:prelim} as it is defined for categories of modules; each such category is indeed equivalent to a category of modules for a finite dimensional algebra. The proof of the theorem follows the plan of earlier proofs: support varieties are used to identify  a one parameter family of nonisomorphic indecomposable objects (of some fixed length), each of complexity at least two. Among the key ingredients are objects generalizing Carlson's $L_{\zeta}$-modules, their action on the module category, and an application of Crawley-Boevey's results to our category setting (Lemma~\ref{lem:wild}).

As mentioned above, finite tensor categories and their finite module categories are known to be equivalent to categories of modules for finite dimensional algebras. We recall the needed details in Remark~\ref{rem:covers}(3) below and use this fact in Lemma~\ref{lem:wild}. One distinction between our work here and past results on finite dimensional algebras is in the hypotheses; finiteness conditions here are imposed on the cohomology of the acting category. Any of its finite module categories then enjoys the representation type consequences of that action. 

Several classes of examples are included to illustrate the theory: In Example~\ref{example:local}, we work in positive characteristic~$p$ and explore some local algebras, focusing particularly on some connected Hopf algebras of dimension $p^3$. In Examples~\ref{example:Verlinde} and~\ref{example:crossed}, we describe how the theory applies to higher Verlinde categories and crossed product categories.

\subsection*{Acknowledgments}
P.\ A.\ Bergh would like to thank the Mathematical Institute at the University of Oxford, where part of this work was done when he visited K.\ Erdmann in October 2024. The visit was supported by LMS Scheme 4 grant no.\ 42355. J.\ Y.\ Plavnik was partially supported by NSF grant DMS-2146392 and by Simons Foundation Award 889000 as part of the Simons Collaboration on Global Categorical Symmetries. S.\ Witherspoon was partially supported by NSF grant 2001163.

\section{Representation type and complexity}\label{sec:prelim}

Throughout this section, we fix an algebraically closed field $k$. We start by recalling the notion of representation type of a finite dimensional $k$-algebra $A$. Unless otherwise specified, whenever we speak of modules, we mean finitely generated left modules. The category of such $A$-modules is denoted by $\mod A$.

The algebra $A$ has \emph{finite representation type} if it has only finitely many isomorphism classes of indecomposable modules. In this case, then, one can classify all the indecomposable modules by simply listing them. If $A$ does not have finite representation type, then it is of \emph{infinite representation type}. A classification of the indecomposable modules may still be possible, if $A$ has \emph{tame representation type}. This means that the representation type is infinite, but for each integer $d \ge 1$ the following holds: almost all the indecomposable modules of dimension $d$ belong to one of a finite number of one-parameter families. Specifically, for each $d \ge 1$ there exists a finite set $B_1, \dots, B_{n_d}$ of $A$-$k[x]$-bimodules satisfying the following: each $B_i$ is free of rank $d$ as a module over $k[x]$, and every $A$-module of dimension $d$ -- except possibly finitely many -- is isomorphic to $B_i \ot_{k[x]} k[x]/(x- \alpha )$ for some $i$ and $\alpha \in k$. Note that since $k$ is algebraically closed, the $k[x]$-modules of the form $k[x]/(x- \alpha )$ are precisely the simple ones.

If $A$ is of infinite representation type but not tame, then a classification of indecomposable modules is beyond reach. Denote by $k \langle x,y \rangle$ the free algebra generated by two non-commuting variables $x$ and $y$. The algebra $A$ has \emph{wild representation type} if there exists an exact embedding from the category of finite dimensional $k \langle x,y \rangle$-modules into $\mod A$, in such a way that the image of an indecomposable $k \langle x,y \rangle$-module is indecomposable in $\mod A$. It is well-known (see \cite[Theorem XIX.1.11]{SiS}) that this is equivalent to the following: for \emph{every} finite dimensional $k$-algebra $B$, there exists such an embedding from $\mod B$ into $\mod A$. In other words, a finite dimensional $k$-algebra is wild if its module category contains the module category of every finite dimensional $k$-algebra. 

It turns out that the two kinds of algebras of infinite representation type that we have discussed are actually all there are. This is a highly nontrivial fact, and was proved by Drozd (see \cite{D} and \cite{CB}). Specifically, he proved that if a finite dimensional algebra over an algebraically closed field is of infinite representation type, then it is either tame or wild, and these are mutually exclusive.

We now turn to the categories that are of interest to us: finite tensor categories and their module categories. For a detailed explanation of the terminology and concepts that follow, we refer to  \cite{EGNO}. A finite tensor category over $k$ is a triple $\left ( \C, \ot, \unit \right )$, where $\C$ is a locally finite $k$-linear abelian category having finitely many isomorphism classes of simple objects, and each of these has a projective cover. Moreover, the tensor product $\ot$ is a bifunctor
\begin{center}
\begin{tikzpicture}
\diagram{d}{3em}{3em}{
 \C \times \C & \C \\
 };
\path[->, font = \scriptsize, auto]
(d-1-1) edge node{$\ot$} (d-1-2);
\end{tikzpicture}
\end{center}
which is bilinear on morphisms and associative up to functorial isomorphisms. Furthermore, the object $\unit \in \C$ is a unit object with respect to the tensor product, and the triple $\left ( \C, \ot, \unit \right )$ is a monoidal category. Finally, the unit object is simple, and $\C$ is rigid in the sense that every object has a left and a right dual. A whole range of important properties follow from the latter, and here are three of them: the category $\C$ is quasi-Frobenius (that is, the projective and the injective objects coincide), the projective objects form a two-sided ideal, and the tensor product is bi-exact (see \cite[Proposition 6.1.3, Proposition 4.2.12, Proposition 4.2.1]{EGNO}).

Now let $\left ( \C, \ot, \unit \right )$ be such a finite tensor category, and $\left ( \M, \ast \right )$ a left module category over $\C$. Thus $\M$ is a locally finite $k$-linear abelian category, together with a bifunctor (the module product)
\begin{center}
\begin{tikzpicture}
\diagram{d}{3em}{3em}{
 \C \times \M & \M \\
 };
\path[->, font = \scriptsize, auto]
(d-1-1) edge node{$\ast$} (d-1-2);
\end{tikzpicture}
\end{center}
which, as with $\ot$, is bilinear on morphisms and associative up to functorial isomorphisms: $(A \ot B) \ast M \simeq A \ast (B \ast M)$ for $A,B \in \C$ and $M \in \M$. Furthermore, the unit object $\unit \in \C$ is a unit object also for the module product, and this product is exact in the first argument. It follows from \cite[Proposition 7.1.6]{EGNO} that the module product is also exact in the second argument, and therefore bi-exact. Moreover, it follows from \emph{loc.~cit}.\ that if $P$ is a projective object in $\M$, then $A \ast P$ is projective for every $A \in \C$. The module category is called \emph{exact} if the latter holds with opposite arguments: if $Q$ is a projective object in $\C$, then $Q \ast M$ is projective for every $M \in \M$. When this holds, then $\M$ is quasi-Frobenius, by \cite[Corollary 7.6.4]{EGNO}, and it also has enough projective objects, by \cite[Lemma 7.6.1]{EGNO}. Finally, we say that $\M$ is \emph{finite} if it has finitely many isomorphism classes of simple objects.

\sloppy The typical example of a finite tensor category is $\mod H$, where $H$ is a finite-dimensional Hopf algebra over $k$. The tensor product is then not necessarily commutative. However, if $G$ is a finite group, then the tensor product in $\mod kG$ is commutative up to natural isomorphism, and in a strong sense, making this a \emph{symmetric} finite tensor category (that is, the natural isomorphism squares to the identity). Note that every finite tensor category is trivially a finite exact module category over itself, with the tensor product as module product. 

In the following remark, we collect some facts that will be of use to us throughout. Let us therefore fix a finite tensor category $\left ( \C, \ot, \unit \right )$, together with a finite exact $\C$-module category $\left ( \M, \ast \right )$.

\begin{remark}\label{rem:covers}
(1) One of the defining properties of $\left ( \C, \ot, \unit \right )$ is that every simple object $S$ has a projective cover. This is defined in \cite{EGNO} as an epimorphism $p \colon Q \longrightarrow S$ with $Q$ projective, and such that the following holds: if $p' \colon Q' \longrightarrow S$ is an epimorphism with $Q'$ projective, then there exists an epimorphism $f \colon Q' \longrightarrow Q$ (necessarily split) with $p' = p \circ f$. One can show that every object admits such a cover, and that the cover map is right minimal (see, e.g.~\cite[\S\S I.2, I.4]{ARS}). Moreover, the same is true for objects in $\M$. 
To see this, take any object $M \in \M$, and let $p \colon P \longrightarrow M$ be an epimorphism with $P$ projective of minimal length. By an argument similar to the proof of \cite[Proposition I.2.1]{ARS}, there exists an object $N \in \M$, together with a commutative diagram
\begin{center}
\begin{tikzpicture}
\diagram{d}{2.5em}{4em}{
 N & P & N \\
 & M \\
 };
\path[->, font = \scriptsize, auto]
(d-1-1) edge node{$f$} (d-1-2)
(d-1-2) edge node{$g$} (d-1-3)
(d-1-1) edge node[below]{$q$} (d-2-2)
(d-1-2) edge node{$p$} (d-2-2)
(d-1-3) edge node{$q$} (d-2-2);
\end{tikzpicture}
\end{center}
in which the morphism $q$ is right minimal. This morphism must be an epimorphism, since $p$ is and $p = q \circ g$. Moreover, by the right minimality of $q$, the composition $g \circ f$ is an isomorphism, so that $N$ is a direct summand of $P$. Thus $N$ is projective, and since we chose $P$ to have minimal length we see that $N \simeq P$, with $f$ and $g$ isomorphisms. Therefore $p$ is right minimal, and thus it also is a projective cover in the sense above.

(2) Let $S_1, \dots, S_n$ be a complete set of representatives for the isomorphism classes of simple objects in $\M$. Their projective covers $P_1, \dots, P_n$ must be indecomposable, and these are all the indecomposable projective objects in $\M$, up to isomorphism. Namely, an induction argument on length shows that every projective object in $\M$ is a direct sum of (some of) the $P_i$.

(3) By the above, the module category $\M$ has projective generators, that is, projective objects $P$ with the property that every projective object in $\M$ is a direct summand of a direct sum of copies of $P$; the object $P_1 \oplus \cdots \oplus P_n$ is one such generator. Denote the finite dimensional $k$-algebra $\End_{\M}(P)^{\op}$ by $\Gamma_P$. By an argument similar to the proof of \cite[Proposition II.2.5]{ARS}, the functor
\begin{center}
\begin{tikzpicture}
\diagram{d}{3em}{6em}{
\M & \mod \Gamma_P \\
 };
\path[->, font = \scriptsize, auto]
(d-1-1) edge node{$\Hom_{\M}(P,-)$} (d-1-2);
\end{tikzpicture}
\end{center}
is an equivalence of $k$-linear abelian categories. In particular, if $P$ and $P'$ are projective generators in $\M$, then $\Gamma_P$ and $\Gamma_{P'}$ are Morita equivalent, and therefore have the same representation type.
\end{remark}

In light of the last remark, we make the following definition.

\begin{definition}
Let $\left ( \C, \ot, \unit \right )$ be a finite tensor category over an algebraically closed field $k$, and $\left ( \M, \ast \right )$ a finite exact $\C$-module category. The \emph{representation type} of $\M$ is the representation type of the finite dimensional $k$-algebra $\End_{\M}(P)^{\op}$, where $P$ is a projective generator in $\M$.
\end{definition}

Thus $\M$ has finite representation type if the number of isomorphism classes of indecomposable objects is finite; a typical example is a fusion category. If $\M$ is tame, then by definition almost all the indecomposable $\End_{\M}(P)^{\op}$-modules of any fixed dimension belong to a finite number of one-parameter families (with $P$ a projective generator in $\M$). For every $\End_{\M}(P)^{\op}$-module $X$, there are inequalities
$$\ell (X) \le \dim_k X \le m \ell (X)$$
where $\ell (X)$ denotes the length, and $m$ is the maximum dimension of the simple modules. Therefore, for the category $\M$ itself, tameness means the following: for any fixed length $d$, there is a finite collection
$$\{ M^1_{\alpha} \}_{\alpha \in k}, \dots, \{ M^{n_d}_{\alpha} \}_{\alpha \in k}$$
of one-parameter families of objects -- not all of which are necessarily of length $d$ -- such that almost all the objects of $\M$ of length $d$ belong to these families (up to isomorphism). Namely, just take all the one-parameter families of $\End_{\M}(P)^{\op}$-modules of dimension at most $dm$, and consider their images in $\M$ under the equivalence.
 
If $\M$ is wild, then for every finite dimensional $k$-algebra $B$, there exists an exact embedding $\mod B \longrightarrow \mod \End_{\M}(P)^{\op}$ that preserves indecomposability. Since $\mod \End_{\M}(P)^{\op}$ is equivalent to $\M$, we obtain such an embedding $\mod B \longrightarrow \M$ as well. In particular, this means that given any finite exact module category $\M'$ over any finite tensor $k$-category $\C'$, there exists such an embedding $\M' \longrightarrow \M$, since $\M'$ is equivalent to the category of modules over $\End_{\M'}(P')^{\op}$, for some (equivalently, every) projective generator $P' \in \M'$. Therefore, when the module category $\M$ is wild, then it contains every finite exact module category.

\begin{remark}\label{rem:hopf}
When $H$ is a finite dimensional Hopf algebra over $k$, then the representation type of the finite tensor category $\mod H$ is the same as the representation type of $H$ as a $k$-algebra, since $H$ is a projective generator in $\mod H$.
\end{remark}

Next, we recall a concept which will play a key role in the main result. As recalled in Remark \ref{rem:covers}(1), every object $M$ of $\M$ admits a projective cover, and therefore also a minimal projective resolution $(P_{\bu}, d_{\bu})$. This is unique up to isomorphism, and is a direct summand in every projective resolution of $M$. The $n$th \emph{syzygy} of $M$ is the image of $d_n$; we denote it by $\Omega_{\M}^n(M)$, and by $\Omega_{\C}^n(X)$ for an object $X \in \C$. The \emph{complexity} of $M$, denoted $\cx_{\M}(M)$, is now defined to be
$$\cx_{\M}(M) = \inf \{ c \in \mathbb{N} \cup \{ 0 \} \mid \exists b \in \mathbb{R} \text{ and } m\in \N \text{ with } \ell_{\M} ( P_n ) \le bn^{c-1} \text{ for all } n \ge m \}$$
where $\ell_{\M} (-)$ denotes the length of objects in $\M$. The following lemma shows that if for a fixed length, there are infinitely many non-isomorphic objects in $\M$ of complexity at least two, then $\M$ is wild. 
This is an application of results of Crawley-Boevey~\cite{CB} in this category setting.

\begin{lemma}\label{lem:wild}
Let $k$ be an algebraically closed field, $\left ( \C, \ot, \unit \right )$ a finite tensor category over $k$, and $\left ( \M, \ast \right )$ a finite exact $\C$-module category. Suppose there exists an integer $d$ for which the following holds: there exist infinitely many non-isomorphic indecomposable objects $M \in \M$ of length $d$ with $\cx_{\M} (M) \ge 2$. Then $\M$ is of wild representation type.
\end{lemma}

\begin{proof}
Let $P_1, \dots, P_n$ be a complete set of representatives of the isomorphism classes of indecomposable projective objects in $\M$, and denote the finite dimensional $k$-algebra $\End_{\M}(P_1 \oplus \cdots \oplus P_n)^{\op}$ by $\Gamma$. By Remark \ref{rem:covers}(3), the category $\mod \Gamma$ is equivalent to $\M$, hence $\Gamma$ is selfinjective since $\M$ is quasi-Frobenius. Moreover, as $P_i \nsimeq P_j$ when $i \neq j$, the algebra $\Gamma$ is basic, and therefore a Frobenius algebra by \cite[Proposition IV.3.9]{SY}. Denote by $\nu$ its Nakayama automorphism, which is unique up to composition with an inner automorphism. By \cite[Proposition IV.3.13 and Theorem IV.8.5]{SY}, the Auslander-Reiten translate $\tau$ is naturally isomorphic to $\Omega_{\Gamma}^2( {_{\nu}(-)})$ as functors on the stable module category $\stmod \Gamma$, where the latter is the functor that maps a module $N$ to $\Omega_{\Gamma}^2( {_{\nu}N} )$. Here ${_{\nu}N}$ is the twisted $\Gamma$-module that equals $N$ as a $k$-vector space, but with module operation $x \cdot m = \nu (x) m$ for $x \in \Gamma$ and $m \in N$.

Since $\M$ is equivalent to $\mod \Gamma$, the assumptions in the statement imply that there exists an integer $d$ for which the following holds: there exist infinitely many non-isomorphic indecomposable $\Gamma$-modules $N$ of length $d$, with $\cx_{\Gamma} (N) \ge 2$. Let $m$ be the maximum of the dimensions (as $k$-vector spaces) of the simple $\Gamma$-modules. Then $\dim_k N \le m \ell_{\Gamma} (N)$ for every $N \in \mod \Gamma$, so there must exist an integer $d'$ for which the following holds: there exist infinitely many non-isomorphic indecomposable $\Gamma$-modules $N$ with $\dim_k N = d'$, and with $\cx_{\Gamma} (N) \ge 2$. For each such module $N$, necessarily $\tau (N) \nsimeq N$. If not, then $\Omega_{\Gamma}^2( {_{\nu}N} )$ and $N$ are isomorphic in $\stmod \Gamma$, and since projective covers commute with automorphism twistings of any kind, it follows that $\Omega_{\Gamma}^2(N) \simeq {_{\nu}N}$. Then by induction  $\Omega_{\Gamma}^{2n}(N) \simeq {_{\nu^n}N}$ for every $n \ge 1$, and consequently $\cx_{\Gamma} (N) \le 1$, a contradiction. Therefore $\tau (N) \nsimeq N$. By~\cite[Theorem D]{CB}, $\Gamma$ is of wild representation type.
\end{proof}

We end this section with a rather long example illustrating the lemma. For finite dimensional Hopf algebras other than group algebras of finite groups, it is not generally understood which ones are of tame representation type. Even for restricted enveloping algebras of Lie algebras, there is at present no general classification; although for example \cite{F2} and \cite{FS} give classifications in many settings, it is usually assumed that the characteristic is not two. Of the known cases, most tame Hopf algebras are special biserial, for example blocks of restricted enveloping algebras of $\mathfrak{sl}(2)$, or Drinfeld doubles of Taft algebras.

\begin{example}\label{example:local} 
Assume that the characteristic of the (algebraically closed) field $k$ is $p \ge 2$, and let $A$ be a finite dimensional local $k$-algebra with two independent generators $x$ and $y$. That is, $x$ and $y$ are in the Jacobson radical $\rad(A)$ and are linearly independent modulo $\rad(A)^2$. There is a one-parameter family $\{ M_{\lambda} \mid \lambda \in k \}$ of two dimensional $A$-modules, with $x$ and $y$ acting on $M_{\lambda}$ via the two matrices 
$$\left ( \begin{array}{cc}
0 & \lambda \\ 0 & 0 
\end{array} \right ) \hspace{3mm} \text{and} \hspace{3mm} \left ( \begin{array}{cc}
0 & 1 \\ 0 & 0 
\end{array} \right )$$
respectively. These modules are indecomposable since, for example, the latter is an indecomposable Jordan block. One also checks that $M_{\lambda} \nsimeq M_{\mu}$ when $\lambda \neq \mu$. Consequently, if $A$ is a Frobenius algebra, and $\cx_A ( M_{\lambda} ) \ge 2$ for each (or infinitely many) $\lambda \in k$, then $A$ is of wild representation type, as shown in the proof of Lemma \ref{lem:wild}.

Suppose, for example, that $A = kG$, the group algebra of a finite group $G$. If $\cx_A ( M_{\lambda} ) =1$, then $M_{\lambda}$ must be a periodic module; see \cite[Theorem 5.10.4]{B}. Since $M_{\lambda}$ is then also periodic when restricted to a subgroup, in particular an elementary abelian subgroup of rank one, the dimension of $M_{\lambda}$ must be divisible by $p$. This cannot happen when $p \ge 3$, and so in this case $\cx_A ( M_{\lambda} ) \ge 2$ for all $\lambda \in k$. The algebra $A$ is therefore wild when $p \ge 3$. More generally, for such $p$, a group algebra $kG$ is wild whenever $G$ has a factor group isomorphic to $C_p \times C_p$. Namely, the group algebra of the latter is isomorphic to $k[x,y]/(x^p, y^p)$, and an algebra having a wild factor algebra is itself wild.

Assume now that $p=2$. The modules constructed above may then have complexity one. For example, consider the algebra 
$$k \langle x,y \rangle / ( x^2, y^2, xyxy-yxyx )$$
which is isomorphic to the group algebra of the dihedral group of order $8$. This algebra is tame (see~\cite[Chapter III]{Er}), and in this case, the modules $M_{\lambda}$ are periodic, of period two (see~\cite[Theorem in Appendix]{Ri}).

For our local algebra $A$ and $p=2$, we now construct a new one-parameter family $\{ N_{\lambda} \mid \lambda \in k \}$ of $A$-modules, this time of dimension three. Note first that both the generators $x$ and $y$ are nilpotent; let us restrict ourselves to the case $y^2=0$. If also $x^2 =0$, then $A$ is special biserial, and therefore tame; see~\cite[Corollary 2.4]{WW}. Assume therefore that $x$ has nilpotency index $a \ge 3$, and also that $x$ and $y$ commute. Now define $N_{\lambda}$ by letting $x$ and $y$ act via the two matrices 
$$\left ( \begin{array}{ccc}
0 & \lambda & 0 \\ 0 & 0 & \lambda \\ 0 & 0 & 0 
\end{array} \right ) \hspace{3mm} \text{and} \hspace{3mm} \left ( \begin{array}{ccc}
0 & 0 &  1 \\ 0 & 0 & 0 \\ 0 & 0 & 0 
\end{array} \right )$$
respectively. When restricted to the subalgebra generated by $x$, this module is indecomposable, hence it is also indecomposable over $A$. Moreover, as for the modules $M_{\lambda}$, it can be shown that $N_{\lambda} \nsimeq N_{\mu}$ when $\lambda \neq \mu$. Finally, it can be shown that, since $p=2$, these modules are not periodic, and that our assumptions on the algebra $A$ then forces them to be of complexity at least two. The algebra is therefore wild. Note that $A$ is isomorphic to a factor algebra of the group algebra of any abelian $2$-group of rank $2$ other than the group of order $4$, and hence such group algebras are also wild.

As a concrete application, let us look at connected Hopf algebras of dimension $p^3$. An almost complete classification of these algebras was given in \cite{NWW}, the exception being those for which the primitive space is a two dimensional abelian restricted Lie algebra. We shall determine the representation type of those algebras from the classification that are isomorphic to enveloping algebras of restricted Lie algebras; they are the ones labelled by C5, C6 and C15 in \cite{NWW}, see \cite[Proposition 6.1]{ESW}. In each case, the Lie algebra $L$ is three dimensional, with a basis $\{ x,y,z \}$ satisfying $[x,y] =z$.

For the algebra C5, the element $z$ is central in $L$, and the $p$-map is trivial. Then $A = U^{[p]}(L)$ has the presentation
$$k \langle x,y \rangle / ( x^p, y^p, (xy-yx)^p, xy-yx \text{ is central} )$$
Let $\mathfrak{a}$ be the ideal of $A$ generated by the element $xy-yx$. Then $A / \mathfrak{a}$ is isomorphic to $k[x,y]/(x^p,y^p)$, which again is isomorphic to the group algebra of $C_p \times C_p$. When $p \ge 3$, this algebra is wild, and then so is $A$. When $p=2$, then $A$ is isomorphic to the group algebra of the dihedral group of order $8$, and therefore tame.

For the algebra C6, the commutator $[x,y]$ is again central. In this case, the algebra $A = U^{[p]}(L)$ has the presentation
$$k \langle x,y \rangle / ( y^p, x^p-(xy-yx), (xy-yx)^p, xy-yx \text{ is central} )$$
As in the previous case, with $\mathfrak{a}$ the ideal of $A$ generated by the element $xy-yx$, the algebra $A / \mathfrak{a}$ is isomorphic to $k[x,y]/(x^p,y^p)$. Thus $A$ is wild when $p \ge 3$. When $p=2$, then by changing the generators from $x,y$ to $x+y,y$ we see that also here $A$ is isomorphic to the group algebra of the dihedral group of order $8$, and therefore tame.

Finally, for the algebra C15, the Lie algebra is $\mathfrak{sl}(2)$, and its restricted enveloping algebra is tame. It has one simple block, and all other blocks are Morita equivalent to Kronecker algebras; see for example \cite{P} or \cite{H}.
\end{example}

\section{The main result}\label{sec:main}

As in the previous section, we fix an algebraically closed field $k$. Furthermore, let us fix a finite tensor category $\left ( \C, \ot, \unit \right )$ over $k$, together with a finite exact $\C$-module category $\left ( \M, \ast \right )$.

For objects $M,N \in \M$, we shall denote the graded $k$-vector space $\oplus_{n=0}^{\infty} \Ext_{\M}^n(M,N)$ by $\Ext_{\M}^*(M,N)$, and similarly for objects in $\C$. The \emph{cohomology ring} of $\C$ is $\Ext_{\C}^*( \unit, \unit )$, which we denote by $\Coh^* ( \C )$. Since the module product is exact in the first variable, there is a homomorphism
\begin{center}
\begin{tikzpicture}
\diagram{d}{3em}{3em}{
\Coh^* ( \C ) & \Ext_{\M}^*(M,M) \\
 };
\path[->, font = \scriptsize, auto]
(d-1-1) edge node{$\varphi_M$} (d-1-2);
\end{tikzpicture}
\end{center}
of graded $k$-algebras, induced by $- \ast M$. This makes $\Ext_{\M}^*(M,N)$ into a right $\Coh^* ( \C )$-module, via $\varphi_M$ followed by Yoneda composition. Similarly, it becomes a left $\Coh^* ( \C )$-module via $\varphi_N$. By \cite[Corollary 2.3]{BPW2}, the left and the right module structures coincide up to a sign for homogeneous elements. In particular, by taking $\M = \C$ and $M = N = \unit$, we see that the cohomology ring $\Coh^* ( \C )$ is graded-commutative; this was proved in a more general setting in \cite[Theorem 1.7]{SA}. It was conjectured by Etingof and Ostrik in \cite{EO} that $\Coh^* ( \C )$ is always finitely generated as a $k$-algebra, and that $\Ext_{\C}^*(X,X)$ is a finitely generated $\Coh^* ( \C )$-module for every object $X \in \C$. This conjecture is still open, and so we therefore make the following definition.

\begin{definition}
The finite tensor category $\left ( \C, \ot, \unit \right )$ satisfies the \emph{finiteness condition} \textbf{Fg} if $\Coh^*( \C )$ is finitely generated, and $\Ext_{\C}^*(X,X)$ is a finitely generated $\Coh^*( \C )$-module for every object $X \in \C$.
\end{definition}

Note that this finiteness condition only involves the objects of $\C$. A priori, then, it has nothing to do with the module category $\M$ and the cohomology of its objects. However, by \cite[Proposition 3.5]{NP}, when \textbf{Fg} holds for $\C$, then $\M$ satisfies a finiteness property on cohomology as well: for every object $M \in \M$, the $\Coh^*( \C )$-module $\Ext_{\M}^*(M,M)$ is finitely generated as a module.

When \textbf{Fg} holds for $\C$, then we obtain a quite powerful theory of support varieties for the objects of $\M$, as explored in \cite{BPW1} and \cite{BPW2}. In the proof of the main result below, we shall make use of these varieties, so we recall here some definitions.

 \begin{definition}
(1) We define
$$\Ho ( \C ) = \left \{ 
\begin{array}{ll}
\Coh^*( \C ) & \text{if the characteristic of $k$ is two,} \\
\Coh^{2*}( \C ) & \text{if not.}
\end{array} 
\right.$$
Furthermore, we define $\m_0$ to be $\Coh^{+} (\C)$, that is, the ideal in $\Ho ( \C )$ generated by the homogeneous elements of positive degrees. 

(2) For an ideal $\az \subseteq \Ho ( \C )$, we denote by $Z ( \az )$ the set of all maximal ideals in $\Ho ( \C )$ containing $\az$.

(3) For an object $M \in \M$, we define $\az_M$ to be the annihilator ideal of $\Ext_{\M}^*(M,M)$ in $\Ho ( \C )$, and the \emph{support variety} $\VM (M)$ to be $\{ \m_0 \} \cup Z( \az_M )$.
\end{definition}

Some comments are in order here. First, note that $\Ho ( \C )$ is commutative in the ordinary sense, since $\Coh^*( \C )$ is graded-commutative. The latter also implies that when the characteristic of $k$ is not $2$, then the homogeneous elements of odd degrees square to zero. So when passing from $\Coh^*( \C )$ to $\Ho ( \C )$, we gain ordinary commutativity and get rid of only nilpotent elements. We do not lose anything in relation to the support variety theory; note that the definition of the finiteness condition \textbf{Fg} is equivalent to the version where we replace $\Coh^*( \C )$ by $\Ho ( \C )$.

Next, note that since $k$ is algebraically closed, the degree zero part of $\Ho ( \C )$, which is $\Hom_{\C} ( \unit, \unit )$, is just $k$, because the unit object is simple. Therefore $\m_0$ is a maximal ideal in $\Ho ( \C )$, and it is the unique graded maximal ideal. Now given any object $M \in \M$, the annihilator ideal $\az_M$ is a graded ideal of $\Ho ( \C )$, and is therefore contained in $\m_0$ as long as $\Ext_{\M}^*(M,M)$ is nonzero. The latter is the case if and only if $M$ itself is nonzero, and so we see that the explicit inclusion of $\m_0$ in the definition of the support variety $\VM (M)$ only matters when $M =0$.

Finally, note that when the assumption \textbf{Fg} holds, then the complexity of the objects of $\M$ is an invariant with several nice properties, as shown in \cite[Theorem 2.7]{BPW2}. For example, the complexity of an object $M \in \M$ is finite and equal to the dimension of the support variety $\VM(M)$, which by definition is $ \Kdim \Ho ( \C ) / \az_M$ (here $\Kdim$ denotes the Krull dimension). In particular, $\cx_{\M}(M)$ is at most $ \Kdim \Ho ( \C )$.

Now to the main result. It states that if \textbf{Fg} holds and $\M$ contains an object of complexity three or more, then $\M$ is of wild representation type. Compare with~\cite[Theorem 4.1]{BS} for selfinjective algebras: If $\Gamma$ is a finite dimensional selfinjective algebra over~$k$ that has a module $M$ with complexity at least three, and $\Ext^*_{\Gamma}(M,M)$ contains a finitely generated subalgebra over which it is a finitely generated module, then $\Gamma$ has wild representation type. Compare also with~\cite[Theorem 2.1]{FW2}, an analogous result for categories of modules of selfinjective algebras satisfying instead a finiteness condition on Hochschild cohomology.

\begin{theorem}\label{thm:main}
Let $k$ be an algebraically closed field, $\left ( \C, \ot, \unit \right )$ a finite tensor category over $k$, and $\left ( \M, \ast \right )$ a finite exact $\C$-module category. If \emph{\textbf{Fg}} holds, and there exists an object $M \in \M$ with $\cx_{\M}(M) \ge 3$, then the category $\M$ is of wild representation type.
\end{theorem}

\begin{proof}
Suppose {\textbf{Fg}} holds and there is an object of complexity at least 3 in $\M$. By the reduction of complexity result~\cite[Theorem 2.7(6)]{BPW2}, there exists an object $M \in \M$ with $\cx_{\M}(M) = 3$. Furthermore, as explained in the paragraph following \cite[Corollary 3.2]{BPW2}, the set of minimal primes lying over the annihilator ideal $\az_M$ in $\Ho ( \C )$ is finite, and they are all graded; let us denote them by $\rz_1, \dots, \rz_n$. Then $\VM (M) = Z( \rz_1 ) \cup \cdots \cup Z( \rz_n )$, and since $\Kdim \Ho ( \C ) / \az_M = 3$, one of these primes, say $\rz_1$, is of coheight three. 

Since $\rz_1$ is graded, it is of the form $( \zeta_1, \dots, \zeta_t )$ for some homogeneous nonzero elements $\zeta_i \in \Ho ( \C )$, say of positive degrees $n_i$. Each $\zeta_i$ corresponds to a nonzero morphism $\Omega_{\C}^{n_i} ( \unit ) \longrightarrow \unit$, which must necessarily be an epimorphism since $\unit$ is a simple object. Denote the kernel of this morphism by $L_{\zeta_i}$. By \cite[Theorem 2.7(3)]{BPW2}, there are equalities
\begin{eqnarray*}
\VM \left ( (L_{\zeta_1} \ot \cdots \ot L_{\zeta_t}) \ast M \right ) & = & Z( \zeta_1 ) \cap \cdots \cap Z( \zeta_t ) \cap \VM (M) \\
& = & Z ( \zeta_1, \dots, \zeta_t )  \cap \VM (M) \\
& = & Z( \rz_1 )  \cap \VM (M) \\
& = & Z( \rz_1 )
\end{eqnarray*}
and so we see that the object $(L_{\zeta_1} \ot \cdots \ot L_{\zeta_t}) \ast M$ is also of complexity $3$. By replacing $M$ with this object, we may without loss of generality suppose that $\VM(M)$ is irreducible, that is, there is a unique minimal (graded) prime $\p$ lying over $\az_M$ in $\Ho ( \C )$.

Consider the ring $\Ho ( \C ) / \p$. Since its Krull dimension is three, the Noether normalization lemma applies: There is a polynomial ring $S = k[w_1,w_2,w_3] \subseteq \Ho ( \C ) / \p$ over which $\Ho ( \C ) / \p$ is a finitely generated module. Moreover, we may suppose that the elements $w_1,w_2,w_3$ are homogeneous, and have positive degrees. The ideal $\p$ is contained in the ideal $\m_0$ generated by all homogeneous elements of positive degree, and so the degree zero part of $\Ho ( \C ) / \p$ is $k$. Now from standard commutative algebra (see, for example, \cite[Chapter 5]{AM}), we obtain some important properties. First of all, since the ring extension is integral, the going-up property holds, and the contraction map $\Spec \Ho ( \C ) / \p \longrightarrow \Spec S$ is surjective. Moreover, since both $S$ and $\Ho ( \C ) / \p$ are integral domains, and $S$ is integrally closed in its field of fractions (it is a unique factorization domain), the going-down property holds as well.

Let $\q_1$ be a prime ideal in $\Ho ( \C ) / \p$ with $\q_1 \cap S = (w_1)$. By the going-up property, there is a chain of primes $\q_1 \subset \q_2 \subset \q_3$ in $\Ho ( \C ) / \p$ with $\q_2 \cap S = (w_1,w_2)$ and $\q_3 \cap S =  (w_1,w_2,w_3)$. Each $\q_i$ is of the form $\p_i / \p$ for a prime ideal $\p_i \subseteq \Ho ( \C )$, and so we obtain a chain 
$$\p \subset \p_1 \subset \p_2 \subset \p_3$$
in $\Spec \Ho ( \C )$. Now, as elements of $\Ho ( \C ) / \p$, write each of $w_1$ and $w_2$ as $w_i = \eta_i + \p$ for some homogeneous element $\eta_i \in \Ho ( \C )$ of positive degree. Then since $w_1 \in \q_1 \setminus \{ 0 \}$ and $w_2 \in \q_2 \setminus \q_1$, we see that $\eta_1 \in \p_1 \setminus \p$ and $\eta_2 \in \p_2 \setminus \p_1$. Choose positive integers $s,t$ such that the homogeneous elements $\eta_1^s, \eta_2^t$ are of the same degree, and let $\alpha \in k$ be any element. Finally, set $\eta = \eta_1^s + \alpha\eta_2^t$ and consider the object 
$$M_{\alpha} = L_{\eta} \ast M$$
where $L_{\eta} \in \C$ is defined in the same way as the objects $L_{\zeta_i}$ from the beginning of the proof. We shall prove that $\cx_{\M}(M_{\alpha}) =2$ for every $\alpha \in k$. 

First note that there are equalities
\begin{equation*}\label{eqn:radicals}
Z ( \az_{M_{\alpha}} ) = \VM ( M_{\alpha} ) = Z ( \eta ) \cap \VM(M) = Z ( \eta ) \cap Z ( \az_M ) = Z \left ( \az_M + ( \eta ) \right ) \tag{$\dagger$} ,
\end{equation*}
where the second equality is by \cite[Theorem 2.7(3)]{BPW2}. By Hilbert's Nullstellensatz (see e.g.~\cite[Theorem 1.6]{E} or~\cite[Theorem 25]{M}), we see that $\sqrt{\az_{M_{\alpha}}} = \sqrt{\az_M + ( \eta )}$, where $\sqrt{\az}$ denotes the radical of an ideal $\az \subseteq \Ho ( \C )$. Therefore the complexity of $M_{\alpha}$, which equals $\Kdim \Ho ( \C ) / \az_{M_{\alpha}}$, is the same as $\Kdim \Ho ( \C ) / ( \az_M + ( \eta ) )$. Furthermore, as $\p$ is the unique minimal prime in $\Ho ( \C )$ lying over $\az_M$, the latter equals $\Kdim \Ho ( \C ) / ( \p + ( \eta ) )$. As this last ring is isomorphic to $\left ( \Ho ( \C ) / \p \right ) / ( w_1^s + \alpha w_2^t )$, it suffices to show that the Krull dimension of $\left ( \Ho ( \C ) / \p \right ) / ( w_1^s + \alpha w_2^t )$ is two. For this, note that since $\Ho ( \C ) / \p$ is an integral domain, and the element $w_1^s + \alpha w_2^t$ is nonzero (the elements $w_1$ and $w_2$ are algebraically independent), Krull's Principal Ideal Theorem (see~\cite[\S 10]{E} or~\cite[Theorem 18]{M}) implies that every minimal prime in $\Ho ( \C ) / \p$ lying over $w_1^s + \alpha w_2^t$ has height one. Combined with the fact that every maximal ideal of $\Ho ( \C ) / \p$ is of height three (see \cite[Corollary 13.4]{E}), we see that the Krull dimension of $\left ( \Ho ( \C ) / \p \right ) / ( w_1^s + \alpha w_2^t )$ must be two. This shows that $\cx_{\M} ( M_{\alpha} ) =2$.

In general, the complexity of a direct sum of objects equals the maximum of the complexities of the summands. Therefore, the indecomposable summands of $M_{\alpha}$ are of complexity at most two, and at least one summand is of complexity two. For each $\alpha \in k$, choose an indecomposable summand $N_{\alpha}$ of $M_{\alpha}$, with $\cx_{\M} ( N_{\alpha} ) =2$. We shall prove that $N_{\alpha} \nsimeq N_{\beta}$ when $\alpha \neq \beta$.

Suppose that $N_{\alpha}$ is isomorphic to $N_{\beta}$, with $\alpha \neq \beta$. Then 
$$Z ( \az_{N_{\alpha}} ) = \VM ( N_{\alpha} ) = \VM ( N_{\beta} ) = Z ( \az_{N_{\beta}} ) , $$ 
giving $\sqrt{\az_{N_{\alpha}}} = \sqrt{\az_{N_{\beta}}}$ as a consequence of Hilbert's Nullstellensatz. Choose a chain $\p_1' \subset \p_2' \subset \p_3'$ in $\Spec \Ho ( \C )$ with $\az_{N_{\alpha}} \subseteq \p_1'$; this is possible since $\Kdim  \Ho ( \C ) / \az_{N_{\alpha}} = 2$. Note that $\az_{N_{\beta}}$ is also contained in $\p_1'$ since $\sqrt{\az_{N_{\alpha}}} = \sqrt{\az_{N_{\beta}}}$. Now since $N_{\alpha}$ is a direct summand of $M_{\alpha}$, the $\Ho ( \C )$-module $\Ext_{\M}^*(N_{\alpha},N_{\alpha})$ is a direct summand of $\Ext_{\M}^*(M_{\alpha},M_{\alpha})$, giving $\az_{M_{\alpha}} \subseteq \az_{N_{\alpha}}$. From (\ref{eqn:radicals}) above and Hilbert's Nullstellensatz, we then obtain 
$$\sqrt{\az_M + ( \eta_1^s + \alpha \eta_2^t )} = \sqrt{\az_{M_{\alpha}}} \subseteq \sqrt{\az_{N_{\alpha}}}$$
and in particular $\eta_1^s + \alpha \eta_2^t \in \sqrt{\az_{N_{\alpha}}}$. Similarly, since $\sqrt{\az_{N_{\alpha}}} = \sqrt{\az_{N_{\beta}}}$, we see that $\eta_1^s + \beta \eta_2^t \in \sqrt{\az_{N_{\alpha}}}$. Consequently, both $\eta_1^s + \alpha \eta_2^t$ and $\eta_1^s + \beta \eta_2^t$ are elements of $\p_1'$. As $\alpha \neq \beta$, both $\eta_1^s$ and $\eta_2^t$, and therefore also the elements $\eta_1$ and $\eta_2$, must belong to $\p_1'$.

From the inclusions $\az_{M_{\alpha}} \subseteq \az_{N_{\alpha}} \subseteq \p_1'$, and the fact that the prime ideal $\p$ is the unique minimal prime containing $\az_{M_{\alpha}}$, we see that $\p \subseteq \p_1'$. The chain $\p_1' \subset \p_2' \subset \p_3'$ therefore corresponds to a chain $\q_1' \subset \q_2' \subset \q_3'$ in $\Spec \left ( \Ho ( \C ) / \p \right )$, with $w_1,w_2 \in \q_1'$ since $\eta_1, \eta_2 \in \p_1'$. Then $\q_1' \cap S$ is a prime ideal in $S$ containing both $w_1$ and $w_2$, and so there is a chain $0  \subset (w_1) \subset ( \q_1' \cap S)$ in $\Spec S$. By the going-down property, there is a corresponding chain $0 \subset \q_0' \subset q_1'$ in $\Spec \left ( \Ho ( \C ) / \p \right )$, with $\q_0' \cap S = (w_1)$. This gives a chain
$$0 \subset \q_0' \subset q_1' \subset \q_2' \subset \q_3'$$ 
in $\Spec \left ( \Ho ( \C ) / \p \right )$, so that $\Kdim \Ho ( \C ) / \p \ge 4$. However, the Krull dimension of $\Ho ( \C ) / \p$ is three, and so we have reached a contradiction. This shows that $N_{\alpha} \nsimeq N_{\beta}$ when $\alpha \neq \beta$.

Finally, consider the length of each object $N_{\alpha}$. First of all, since this object is a summand of $M_{\alpha}$, we obtain the inequality $\ell_{\M} ( N_{\alpha} ) \le \ell_{\M} ( M_{\alpha} )$. Now recall that $M_{\alpha} = L_{\eta_1^s + \alpha \eta_2^t} \ast M$, where $L_{\eta_1^s + \alpha \eta_2^t}$ is the kernel of an epimorphism $\Omega_{\C}^n ( \unit ) \longrightarrow \unit$ in $\C$, with $n$ the degree of $\eta_1^s + \alpha \eta_2^t$. The module product $- \ast M$ is exact, so from the exact sequence
\begin{center}
\begin{tikzpicture}
\diagram{d}{3em}{3em}{
0 & L_{\eta_1^s + \alpha \eta_2^t} & \Omega_{\C}^n ( \unit ) & \unit & 0 \\
 };
\path[->, font = \scriptsize, auto]
(d-1-1) edge (d-1-2)
(d-1-2) edge (d-1-3)
(d-1-3) edge (d-1-4)
(d-1-4) edge (d-1-5);
\end{tikzpicture}
\end{center}
in $\C$, we obtain an exact sequence
\begin{center}
\begin{tikzpicture}
\diagram{d}{3em}{3em}{
0 & M_{\alpha} & \Omega_{\C}^n ( \unit ) \ast M & M & 0 \\
 };
\path[->, font = \scriptsize, auto]
(d-1-1) edge (d-1-2)
(d-1-2) edge (d-1-3)
(d-1-3) edge (d-1-4)
(d-1-4) edge (d-1-5);
\end{tikzpicture}
\end{center}
in $\M$. Consequently, we see that $\ell_{\M} ( N_{\alpha} ) \le \ell_{\M} ( \Omega_{\C}^n ( \unit ) \ast M )$ for every $\alpha \in k$. As $k$ is infinite, this implies that there exists an integer $d$ with the property that $\ell_{\M} ( N_{\alpha} ) = d$ for infinitely many elements $\alpha \in k$. Therefore, by Lemma \ref{lem:wild}, the module category $\M$ is of wild representation type.
\end{proof}

When \textbf{Fg} holds, by \cite[Theorem 4.1]{BPW1}, the complexity of the unit object $\unit \in \C$ equals the Krull dimension of $\Ho ( \C )$, and $\cx_{\C} (X) \le \cx_{\C} ( \unit )$ for every object $X \in \C$. Consequently, we see that $\Kdim \Ho ( \C ) \ge 3$ if, and only if, there exists an object $X \in \C$ with $\cx_{\C} (X) \ge 3$. We therefore obtain the following corollary of Theorem~\ref{thm:main}.

\begin{corollary}\label{cor:main}
Let $k$ be an algebraically closed field and $\left ( \C, \ot, \unit \right )$ a finite tensor category over $k$. If \emph{\textbf{Fg}} holds, and $\Kdim \Ho ( \C ) \ge 3$, then the category $\C$ is of wild representation type.
\end{corollary}

We end by illustrating the above theorem and corollary with examples of higher Verlinde categories and crossed product categories.

\begin{example}\label{example:Verlinde}
(Higher Verlinde categories.)
Let $k$ be an algebraically closed field of positive characteristic~$p$,
and let~$n$ be a positive integer.
Let $\C = \Ver_{p^n}$ be a Verlinde category (see for example~\cite{BE,CEO} for
definitions).
These are incompressible finite tensor categories of moderate growth.
Their importance is due to the result of
Coulembier, Etingof, and Ostrik~\cite{CEO} that every pretannakian category
of moderate growth has a tensor functor to an incompressible category of
moderate growth. In positive characteristic, the latter are conjectured to be
precisely the subcategories of these Verlinde categories (including $\Ver_{p^{\infty}}$)
as classified in~\cite[Corollary 4.6]{BEO}.
The $n=1$ case $\Ver_p$ is semisimple for all~$p$, being the semisimplification of the
category of representations of $\Z/p\Z$.
Benson and Etingof~\cite{BE} proved generally that $\coh^*(\Ver_{p^n})$ 
is finitely generated of Krull dimension $n-1$.
By our Corollary~\ref{cor:main} then, $\Ver_{p^n}$ is of wild representation type
for all $p$ and all $n\geq 4$. 
Likewise, by Theorem~\ref{thm:main}, any finite exact $\Ver_{p^n}$-module category having an object of complexity at least 3 is of wild representation type, as long as $n\geq 4$.
The corollary does not apply to the nonsemisimple
Verlinde categories $\Ver_{p^2}$ and $\Ver_{p^3}$.
\end{example}

\begin{example}\label{example:crossed}
(Crossed product categories.)
In~\cite{BPW3}, we studied crossed product categories $\C\rtimes G$ where 
$\C$ is a finite tensor category and $G$ is a finite group acting by tensor autoequivalences.
As abelian categories, these are equivalent to the Deligne products $\C\boxtimes\Vect_G$, where $\Vect_G$ is the category of $G$-graded finite dimensional vector spaces over $k$.
The crossed product category $\C\rtimes G$ is a finite exact $\C$-module category.
Thus Theorem~\ref{thm:main} applies to $\C\rtimes G$. 
\end{example}



\end{document}